\documentclass[12pt,a4paper]{amsart}
\usepackage{amsfonts}
\usepackage{amsthm}
\usepackage{amsmath}
\usepackage{amssymb}
\usepackage{enumitem,kantlipsum}
\newtheorem{theorem}{Theorem}
\newtheorem{lemma}[theorem]{Lemma}
\newtheorem{conj}[theorem]{Conjecture}

\newcommand{\R}{\mathbb R}
\newcommand{\C}{\mathbb C}

\usepackage[colorlinks=true, allcolors=black]{hyperref}
\DeclareMathOperator{\Oct}{\mathbb{O}}
\DeclareMathOperator{\HH}{\mathbb{H}}

\theoremstyle{definition}

\newtheorem*{definition*}{Definition}
\newtheorem{example}[theorem]{Example}

\theoremstyle{remark}

\numberwithin{equation}{section}
\DeclareMathOperator{\Image}{Im}

\DeclareMathOperator{\ssl}{sl}

\newenvironment{proof-sketch}{\noindent{\bf Sketch of Proof}\hspace*{1em}}{\qed\bigskip}
\newenvironment{proof-idea}{\noindent{\bf Proof Idea}\hspace*{1em}}{\qed\bigskip}
\newenvironment{proof-of-lemma}[1]{\noindent{\bf Proof of Lemma #1}\hspace*{1em}}{\qed\bigskip}
\newenvironment{proof-of-prop}[1]{\noindent{\bf Proof of Proposition #1}\hspace*{1em}}{\qed\bigskip}
\newenvironment{proof-of-thm}[1]{\noindent{\bf Proof of Theorem #1}\hspace*{1em}}{\qed\bigskip}
\newenvironment{proof-attempt}{\noindent{\bf Proof Attempt}\hspace*{1em}}{\qed\bigskip}

\thanks{We would like to thank Alexei Kanel-Belov for fruitful discussions related to this paper.}
\begin{document}

\title[Polynomial evaluations on Jordan algebras]{Evaluations of multilinear polynomials on low rank Jordan algebras}

\author{Sergey Malev
  \and
  Roman Yavich
  \and
  Roee Shayer
}

\newcommand{\Addresses}{{
  \bigskip
  \footnotesize

  Sergey Malev (Corresponding author), \textsc{Department of Mathematics, Ariel University,
Ariel, Israel}\par\nopagebreak
  \textit{E-mail address}, Sergey Malev: \texttt{sergeyma@ariel.ac.il}

  \medskip

  Roman Yavich, \textsc{Department of Mathematics, Ariel University,
Ariel, Israel}\par\nopagebreak
  \textit{E-mail address}, Roman Yavich: \texttt{romany@ariel.ac.il}

  \medskip

  Roee Shayer, \textsc{Department of Mathematics, Ariel University,
Ariel, Israel}\par\nopagebreak
  \textit{E-mail address}, Roee Shayer: \texttt{roeeshayer10@gmail.com}

}}

\maketitle

\begin{abstract}
In this paper we prove the generalized Kaplansky conjecture for Jordan algebras of the type $J_n$,  in particular for self-adjoint $2\times 2$ matrices over $\R$, over $\C$, $\HH$ and $\Oct$. In fact, we prove that the image of multilinear polynomial must be either $\{0\}$, $\mathbb{R}$, the space $V$ of pure elements, or $J_n$.  
\end{abstract}

\section{Introduction}

Images of polynomials evaluated on algebras is
one of the most important areas in the modern algebra. 
This question was considered by Bre{\v s}ar and his school, and we recommend the reader to consider very interesting and deep works \cite{Br,BK,V}.
For a more complete survey we recommend \cite{BMRY}, and for a similar survey on related topics in group theory one can see \cite{GKP}.

One of the central conjectures regarding possible evaluations of multilinear polynomials on matrix algebras was attributed to Kaplansky and formulated by L'vov in \cite{Dn}:
\begin{conj}[L'vov-Kaplansky] \label{kapl}Let $K$ be an infinite field and $n\in \mathbb{N}$. If $p$ is a noncommutative multilinear  $K$-polynomial then the image of $p$ evaluated on the matrix algebra $M_n(K)$ is a vector subspace of $M_n(K)$.

\end{conj}

It is well-known (\cite{BMR1,BMR2,BMR3,BMR4,BMRY,M1,M2}) that this conjecture can be reformulated as follows:
\begin{conj}\label{Kapl2} If $p$ is a noncommutative  multilinear polynomial evaluated on the matrix algebra
$M_n (K)$, then $\Image p$ is either $\{0\}, K, \ssl_n (K),$ or $M_n (K).$ Here $K$ indicates the set
of scalar matrices and $\ssl_n (K)$ is the set of  matrices with trace equal to zero.
\end{conj}

When $n=2$, for the case of $K$ being quadratically closed Conjecture~\ref{Kapl2} was proved in \cite{BMR1}. In \cite{M1} it was proved for the case of $K=\R$, and an interesting result was obtained for arbitrary fields.

For $n>2$ this question was considered in \cite{BMR2,BMR3,BMRY} and partial results were obtained.\newline

An interesting question is whether Conjecture \ref{kapl} holds in algebras other than matrix algebras.
Note that matrix algebras are simple and finite-dimensional. Thus, we come to the following generalized conjecture:
\begin{conj}\label{gen}
Let $F$ be an infinite field. If $A$ is a simple  finite-dimensional algebra 	over $F$ then the evaluation of any multilinear polynomial over $A$ is a vector subspace of $A$.
\end{conj}

The condition of being a simple algebra is essential, since Conjecture \ref{gen} fails for the (non-simple) Grassman algebra, as mentioned in \cite{M2}.\newline

 A well-known theorem of Frobenius says that the only associative finite-dimensional division algebras over $\mathbb{R}$ are the field of real numbers $\mathbb{R}$, the field complex numbers $\mathbb{C}$ and the algebra of quaternions $\mathbb{H}$.\newline
Not requiring the condition of being associative, a result of Adams (\cite{ad}) says that a real finite-dimensional division algebra must have dimension 1, 2, 4, or 8.
Moreover, there are only four finite-dimensional real composition algebras. These are $\mathbb{R}$, $\mathbb{C}$, $\mathbb{H}$ and the algebra $\mathbb{O}$ of octonions.\newline

In \cite{M2} the problem of evaluation of polynomials was considered for the algebra of quaternions with the Hamilton multiplication and it was shown that any evaluation of a multilinear polynomial  is a vector space.
 Thus, as $\mathbb{O}$ is the only $8$-dimensional real division algebra known at the time of writing, it is important to consider the corresponding problem for evaluation of polynomials on $\mathbb{O}$. We consider the following conjecture on a paper being prepared:
\begin{conj}\label{oct}
If $p(x_1,\dots,x_n)$ is a non-associative  non-commutative multilinear  polynomial over $\mathbb{R}$, then the image of $p(x_1,\dots,x_n)$ evaluated on $\mathbb{O}$ is a real  vector subspace of $\mathbb{O}$.
\end{conj}

In \cite{MP} a problem of evaluations of non-associative commutative polynomials on the Rock-Paper-Scissors algebra was considered. Polynomials are commutative since the Rock-Paper-Scissors algebra is commutative.
In the present paper we consider non-associative commutative multilinear  polynomials evaluated on Jordan algebras. Jordan algebras were first introduced by Pascual Jordan (1933). 
For more information about Jordan algebras see \cite{ZSSS}. 

\section{Preliminaries}

Let us recall here the definition of Jordan algebras:
\begin{definition*}
A Jordan algebra is a non-associative algebra whose multiplication $\circ$ satisfies the following axioms: 
\begin{enumerate}
    \item $x\circ y=y\circ x$ (commutativity)
    \item $(x\circ y)\circ (x\circ x)=x\circ (y\circ (x\circ x))$ (Jordan identity).
\end{enumerate}
\end{definition*}
According to the first axiom, all Jordan algebras are commutative, and thus we will consider evaluations of commutative polynomials. 
Recall also that for any associative algebra $(A,\cdot,+)$ over a field of characteristic not equal to $2$ one can construct so called {\it special Jordan algebra} with product $\circ$ defined as $$x\circ y=\frac{x\cdot y+y\cdot x}{2}.$$
Note that MacDonald's Theorem (see \cite{LM, MC}) says that if a polynomial identity in three
variables $x, y, z$ which is linear in $z$ holds for all special
Jordan algebras, it holds for all Jordan algebras. Jordan algebras which are not special are called {\it exceptional Jordan algebras.}

An important class of exceptional Jordan algebras are {\it Hermitian Jordan algebras}. These are algebras defined on self-adjoint subsets of associative algebras with standard Jordan multiplication (defined in the same way as in special Jordan algebras). Usually these sets do not form an associative algebra since they are not closed under the standard associative multiplication, and so a corresponding Jordan algebra is exceptional. Nevertheless these sets are closed under Jordan multiplication.

We will consider evaluations of arbitrary degree multilinear polynomials on algebras of type $J_n$ which we define as follows: the base of $J_n$ is the set 
$\{e_0=1,e_1,\dots,e_{n-1}\}$, and the product $\circ$ is defined as $1\circ x=x\circ 1=x$ for any $x$, $e_i\circ e_j=\delta_{ij}$ where $\delta_{ij}$ is the Kronecker delta. Note that type $J_n$ is a particular case of the Jordan algebra of a nondegenerate symmetric bilinear form.
Note that the two-dimensional algebra $J_2$ is isomorphic to the split complex algebra. This is a simple case: $J_2$ is commutative, and associative, so all polynomial evaluations on $J_2$ are or $\{0\}$ or all the algebra.

There are at least four different Hermitian Jordan algebras of this type which are of interest for us:
\begin{enumerate}
    \item The three-dimensional algebra $J_3$ of self-adjoint $2\times 2$ real matrices with standard Jordan product $x\circ y=\frac{xy+yx}{2}$.
Indeed, one can take the following basis: $1$ is the identity matrix, 
$e_1$ is $\begin{pmatrix}
1 & 0\\
0 & -1
\end{pmatrix}$,
and
$e_2$ is $\begin{pmatrix}
0 & 1\\
1 & 0
\end{pmatrix}$.
It is not difficult to see that $e_1\circ e_1=e_2\circ e_2=1$ and $e_1\circ e_2=e_2\circ e_1=0$.
\item The four-dimensional algebra $J_4$ of self-adjoint $2\times 2$ complex matrices with standard Jordan product $x\circ y=\frac{xy+yx}{2}$.
Indeed, here we can take a basis $\{e_0=1, e_1,e_2,e_3\}$, where first three elements are the same as in $J_3$ and 
$e_3=\begin{pmatrix}
0 & i\\
-i & 0
\end{pmatrix}$.
One can check that $e_i\circ e_j=\delta_{ij}1$ for any $i,j>0$.
\item The six-dimensional algebra $J_6$ of  self-adjoint $2\times 2$ quaternionic matrices with standard Jordan product $x\circ y=\frac{xy+yx}{2}$.
Indeed, here we can take a basis $\{e_0=1, e_1,\dots,e_5\}$, where the first four elements are the same as in $J_4$ and 
$e_4=\begin{pmatrix}
0 & j\\
-j & 0
\end{pmatrix}$,
$e_5=\begin{pmatrix}
0 & k\\
-k & 0
\end{pmatrix}$.

One can check that $e_i\circ e_j=\delta_{ij}1$ for any $i,j>0$ here as well.
\item The ten-dimension algebra $J_{10}$ of  hermitian $2\times 2$ octonionic matrices with standard Jordan product $x\circ y=\frac{xy+yx}{2}$.
Here we can take a basis $\{e_0=1, e_1,\dots,e_9\}$, where first six elements are the same as in $J_6$ and 
$e_6=\begin{pmatrix}
0 & l\\
-l & 0
\end{pmatrix}$,
$e_7=\begin{pmatrix}
0 & il\\
-il & 0
\end{pmatrix}$,
$e_8=\begin{pmatrix}
0 & jl\\
-jl & 0
\end{pmatrix}$,
$e_9=\begin{pmatrix}
0 & kl\\
-kl & 0
\end{pmatrix}$.
\end{enumerate}
The purpose of this paper is to prove the following Theorem (and thus to show that low rank Jordan algebras fulfill the L'vov-Kaplansky Conjecture):
\begin{theorem}\label{main}
Let $p$ be any commutative non-associative polynomial. Then its evaluation on $J_n$ is either $\{0\}$, or $\mathbb{R}$ (i.e. one-dimensional subspace spanned by the identity element), the space $V$ of pure elements ($(n-1)$-dimensional vector space  $\left< e_1,\dots,e_{n-1}\right>$) , or $J_n$. 
\end{theorem}
The complete proof of this Theorem gives us a hope to generalize this conjecture to Jordan algebras of higher rank, in particular for the famous Albert algebra (i.e. a 27-dimensional algebra of self-adjoint $3\times 3$ matrices with octonion entries) which occurs in Physics and such a result would be very applicable. More generally, classifying possible evaluations of polynomials on
algebras is part of investigating solutions of polynomial equations
which is fundamental to Algebra.

\section{Evaluations of multilinear polynomials on $J_n$}
Note that any polynomial evaluation (in particular, multilinear polynomial evaluation) is closed under each automorphism and thus contains each element with all its orbit under the automorphism group of the algebra.
Let us start by investigating the orbit of an element.
\begin{lemma}\label{orbit}
The orbit of any element $a+v$ (where $a\in\R, v\in V$) is the set of all elements $a+u$ where $||u||=||v||$. Here $||u||$ is the euclidean norm of the vector $u$, given by the scalar product defined by 
$\langle e_i,e_j \rangle=\delta_{ij}.$
\end{lemma}
\begin{proof}
First let us see that for any pure elements $u$ and $v$, $u\circ v=\langle u, v \rangle$. Indeed, if $u=\sum\limits_{i=1}^{n-1}c_ie_i$ and 
$v=\sum\limits_{i=1}^{n-1}d_ie_i,$
then 
$$u\circ v=\sum_{i,j=1}^{n-1}(c_ie_i)\circ (d_je_j)=\sum_{i,j=1}^{n-1}c_id_j \delta_{ij}=\sum_{i=1}^{n-1}c_id_i=\langle u, v \rangle.$$ In particular, $u^2=||u||^2$.

Let us show that if two pure elements have the same norm $||u||=||v||$ then $u\sim v$ (i.e. $u$ is an automorphic image of $v$.)
Let $j_1,\dots,j_{n-1}$ be arbitrary unit pairwise perpendicular pure elements (i.e. any orthonormal basis of $V$). Consider the linear map $\Phi$ defined on the standard base of $J_n$ as follows:
$\Phi(1)=1$ and for each $1\le i\le n-1$ 
$\Phi(e_i)=j_i.$
Then obviously $\Phi$ is an automorphism.
In particular, we can take $j_1=\frac{u}{||u||}$, and thus $u$ is similar to the element $||u||e_1$. By the same reason, $v$ is similar to the element $||v||e_1$ which is the same if $||u||=||v||$. Therefore, $u\sim v$.
Now, it is not difficult to see that $a+u\sim a+v$ and thus the orbit of $a+v$ contains all elements $a+u$ for which $||v||=||u||$. One can see that it cannot contain any other elements by two reasons:
\begin{enumerate}
    \item a pure element cannot be similar to non-pure, thus the orbit of an element can contain only elements with the same real part.
    \item if two elements with the same real part are equivalent then their pure parts must have the same norm.
\end{enumerate}
Indeed, we can show (1) as follows. Assume $u\sim b+v$ and $b\neq 0$. Thus, $u^2=||u||^2\sim (b+v)^2=b^2+v^2+2bv$ which must be real. Note that $b^2+v^2$ is real, but $2bv\in V$ can be real only if it equals $0$. This happens only if $b=0$ or $v=0$. Hence, if $b\ne 0$ then $v=0$ and $u\sim b$ and thus $u-b\sim 0$ which cannot happen. A contradiction.

Now, if we assume that  $u\sim v$, note that $u\sim ||u||e_1$ and $v\sim ||v||e_1$ and thus, $e_1\sim c e_1$ for $c=\frac{||u||}{||v||}$. Thus $1=e_1^2\sim c^2 e_1^2=c^2$ and $c^2=1$ which happens only if $c=\pm 1$. The number $c=\frac{||u||}{||v||}$ is positive and thus equals $1$. Therefore (2) holds as well.
\end{proof}

Now let us consider basic evaluations of a multilinear polynomial $p$, i.e. values $p(z_1,z_2,\dots,z_m)$ where each $z_i$ is a basic element. Moreover, we will let at this step $z_i$ to be arbitrary real or pure elements.
\begin{lemma}\label{basic}
Let $p(x_1,\dots,x_m)$ be an arbitrary commutative non-associative multilinear polynomial. Then for any elements $x_i\in J_n$ if each $x_i$ is either real or pure, the value of $p$ is either real or pure. Moreover, if the number of pure elements among the $x_i$ is odd, the value is pure, and if it is even, the value is real.
\end{lemma}
\begin{proof}
Any polynomial $p$ is a sum of monomials, and if $p$ is multilinear, then each monomial of $p$ has degree $1$ in each variable.
It is enough to show that the statement holds for any monomial, i.e. any product of our elements (regardless of the order and brackets).
We will prove this by induction on the degree $m$ of the polynomial $p$. First assume that $m=1$ or $2$. Then it is not difficult to see that the statement of the Lemma holds.
Now assume that it holds for each multilinear polynomial of degree $k<m$. Let us show that it holds for degree $m$ as well. 

Any commutative non-associative monomial can be uniquely written as a product of two submonomials, each has degree less than $m$. And according to the assumption of induction the statement holds for both submonomials.
Note, any product of two pure elements is real, any product of the real element and pure is pure, and, of course, any product of two real elements is real. Therefore this assumption holds also for the $m$-degree monomial and thus holds for arbitrary multilinear polynomial which is the sum of such monomials.
\end{proof}
Now we are ready to prove the main Theorem.

\begin{proof-of-thm}{\ref{main}}
If we consider basic evaluations of $p$, according to Lemma \ref{basic} they can be real or pure only. Therefore, we have $4$ cases:
\begin{enumerate}
    \item all basic evaluations vanish. It this case $p$ is PI. 

    \item all basic evaluations are real. In this case $\Image p=\R.$

    \item all basic evaluations are pure. In this case, according to Lemma~\ref{orbit}, $\Image p=V.$

    \item  there are and real and pure nonzero basic evaluations. 
\end{enumerate}
We should consider the last case, since all other cases are done. Let us show, in this case $\Image p=J_n.$

There are basic elements 
$x_1,\dots,x_m$ and $y_1,\dots,y_m$ such that $$p(x_1,\dots,x_m)=k\in\R\setminus\{0\}\  \text{and}\ p(y_1,\dots,y_m)=v\in V\setminus\{0\}.$$ Let us consider the following $m+1$ polynomial functions depending on $z_1,\dots,z_m$:
$A_1=p(x_1,x_2,\dots,x_m); A_2=p(z_1,x_2,\dots,x_m);
A_3=p(z_1,z_2,x_3,\dots,x_m);\dots; A_{m+1}=p(z_1,\dots,z_m)$.
Note that $A_1$ is a constant polynomial taking only one possible value (which is a nonzero scalar), for any $i\ \Image A_i\subseteq\Image A_{i+1}$ and $\Image A_{m+1}=\Image p$ includes nonscalar values. Therefore, there exists $i$ such that   $\Image A_i\subseteq\R$ and  $\Image A_{i
+1}\not\subseteq\R$. 
Consider generic elements
$z_1,z_2,\dots,z_m$. 
For them $A_i(z_1,z_2,\dots,z_m)$ is a nonzero real and $A_{i+1}(z_1,z_2,\dots,z_m)$ is not real.
We can simplify it, and write in the following way: there exist some collection $r_1,r_2,\dots,r_m,r_i^*$  of elements of $J_n$ such that 
$$p(r_1,r_2,\dots,r_m)=r\in\R\setminus\{0\}\ \text{and}\ 
 p(r_1,r_2,\dots,r_{i-1}, r_i^*, r_{i+1},\dots,r_m)\notin\R.$$ Assume that 
 $p(r_1,r_2,\dots,r_{i-1}, r_i^*, r_{i+1},\dots,r_m)=a+v$ for $a\in\R$ and $v\in V$. Then $v\neq 0$. If $a=c\cdot p(r_1,r_2,\dots,r_m)$ we can take 
$\tilde r_i=r_i^*-cr_i$, and 
$$p(r_1,r_2,\dots,r_{i-1}, \tilde r_i, r_{i+1},\dots,r_m)=v\in V\setminus\{0\}.$$
Note that for arbitrary real numbers $x$ and $y$ we have an element 
$$p(r_1,r_2,\dots,r_{i-1}, xr_i+yr_i\tilde r_i, r_{i+1},\dots,r_m)=xr+yv.$$
Thus, according to Lemma \ref{orbit} the evaluation of $p$ must contain all the elements of $J_n.$
\end{proof-of-thm}

\section{Examples}
Now let us show that options described in Theorem \ref{main} can be achieved as multilinear evaluations of $J_n$ for $n\ge 3$.
\begin{example}
\begin{enumerate}
\item
One of the most usefull non-associative polynomials is the associator $p(x,y,z)=(xy)z-x(yz)$. Let us check its evaluation on $J_n$. For that we consider its evaluations on basic elements. If one of the parameters is real, then the value of the associator  is zero. Therefore it is enough to check its value when all the parameters are pure. In this case the number of pure parameters is odd, and thus the value is pure. Note that it can be nonzero, in particular 
$p(e_1,e_1,e_2)=e_1^2\circ e_2-e_1(e_1e_2)=e_2-0=e_2.$ Therefore, $\Image p=V$.
\item We can construct an example of a central polynomial as well: if $p(x,y,z)$ is the associator, then $p(x_1,y_1,z_1)\circ
p(x_2,y_2,z_2)$ takes values which are the products of any two pure elements, and this set is $\R$.
\item Any finite-dimensional algebra is a PI algebra, however let us construct a multilinear example of a PI. In particular we can take the following:
$p(p(x_1,y_1,z_1)\circ p(x_2,y_2,z_2),y_3,z_3)$ where $p$ is the same associator as in the previous examples. In this case one of the parameters of the associator is real and since algebra is commutative all its values are zeros only.
\item Of course, evaluation can be the whole algebra, in particular for the polynomial $f(x)=x$.
\end{enumerate}
\end{example}


\Addresses

\end{document}